\documentclass[12pt,reqno]{amsart}
\usepackage{amsopn,amssymb,mathrsfs,a4wide,amsmath,color,graphicx,enumitem,mathtools}
\usepackage{changes}
\usepackage{physics}
\input xy
\xyoption{all}

\colorlet{Changes@Color}{red}
\newtheorem{thm}{Theorem}

\newtheorem{lem}[thm]{Lemma}

\newtheorem{prob}[thm]{Problem}

\newtheorem*{mainthm}{Main Theorem}

\theoremstyle{definition}

\newtheorem{rem}[thm]{Remark}

\newcommand{\C}{\mathbb{C}}                                                 
\newcommand{\e}{\varepsilon}
\newcommand{\free}[1]{\ensuremath{\mathcal{F}({#1})}}                                              
\newcommand{\FF}{\mathcal{F}}
\renewcommand{\ip}[2]{\ensuremath{\left\langle{#1},{#2}\right\rangle}}

\renewcommand{\geq}{\geqslant}
\renewcommand{\leq}{\leqslant}
\newcommand{\lint}[4]{\ensuremath{\int_{#1}^{#2}{#3}\:\mathrm{d}{#4}}}  
\newcommand{\Lipp}{\mathrm{Lip}}
\newcommand{\Lipz}[1]{\ensuremath{\Lip_0({#1})}}
\newcommand{\map}[3]{\ensuremath{{#1}\colon{#2}\longrightarrow{#3}}}             
\newcommand{\N}{\mathbb{N}}
\newcommand{\n}[1]{\norm{#1}}
\newcommand{\ndot}{\n{\,\cdot\,}}                                       
\newcommand{\pn}[2]{\n{#1}_{#2}}                          
\newcommand{\R}{\mathbb{R}}                                                 
\newcommand{\res}[1]{\ensuremath{\!\!\upharpoonright_{#1}}}                 
\newcommand{\set}[2]{\ensuremath{\left\{{#1}\colon {#2}\right\}}}         
\newcommand{\tri}{{\displaystyle |\kern-.9pt|\kern-.9pt|}}
\newcommand{\tribig}{{\displaystyle \Bigg|\kern-.9pt\Bigg|\kern-.9pt\Bigg|}}
\newcommand{\tn}[1]{\ensuremath{\tri{#1}\tri}}    
\newcommand{\tnbig}[1]{\ensuremath{\tribig{#1}\tribig}}                                                   
\newcommand{\tndot}{\ensuremath{\tri\cdot\tri}}                                                 

\newcommand{\Lipf}{\ensuremath{\Lip_{0,F}\left(rB_{X}\right)}}

\DeclareMathOperator{\aspan}{span}                                        
\DeclareMathOperator{\Lip}{Lip}                                                                                     
\DeclareMathOperator{\conv}{co}                                          

\DeclareMathOperator{\dist}{dist}

\allowdisplaybreaks

\begin{document}
\title[Lipschitz-free spaces over compact subsets]{Lipschitz-free spaces over compact subsets\\ of superreflexive spaces are weakly\\ sequentially complete}

\author[T. Kochanek]{Tomasz Kochanek}
\address{Institute of Mathematics, Polish Academy of Sciences, \'Sniadeckich 8, 00-656 Warsaw, Poland\, {\rm and}\, Institute of Mathematics, University of Warsaw, Banacha~2, 02-097 Warsaw, Poland}
\email{tkoch@impan.pl}

\author[E. Perneck\'a]{Eva Perneck\'a}
\address{Institute of Mathematics, Polish Academy of Sciences, \'Sniadeckich 8, 00-656 Warsaw, Poland\, {\rm and}\, Faculty of Information Technology, Czech Technical University in Prague, Th\'akurova 9, 160 00, Prague 6, Czech Republic}
\email{e.pernecka@impan.pl}

\subjclass[2010]{Primary 46B03, 46B20, 54E35}
\keywords{Lipschitz-free space, weakly sequentially complete}

\begin{abstract}
Let $M$ be a~compact subset of a~superreflexive Banach space. We prove that the~Lipschitz-free space $\mathcal{F}(M)$, the predual of the~Banach space of Lipschitz functions on $M$, has the Pe\l czy\'nski's property ($V^\ast$). As a~consequence, the~Lipschitz-free space $\mathcal{F}(M)$ is weakly sequentially complete.

\end{abstract}
\maketitle	

\section{Introduction}

\noindent
By the~Aharoni's result \cite{aharoni}, if a~metric space $M$ contains a~bilipschitz copy of $c_0$, then the~Lipschitz-free space $\free{M}$ contains an~isomorphic copy of every separable Banach space. In \cite{DF}, Dutrieux and Ferenczi asked about the~converse in the~case of Banach spaces, that is, whether for a~Banach space $X$ whose Lipschitz-free space $\free{X}$ is a~universal separable Banach space, $X$ contains a~bilipschitz copy of $c_0$. C\'uth, Doucha and Wojtaszczyk addressed this question in \cite{CDW} and provided partial progress, which we cite in Theorem \ref{CDW_theorem} below.

A sequence $(x_n)_{n=1}^{\infty}$ in a~Banach space $X$ is \emph{weakly Cauchy} if the~sequence $(\ip{x_n}{x^\ast})_{n=1}^\infty$ is convergent for every $x^\ast\in X^\ast$. A~Banach space $X$ is called \emph{weakly sequentially complete} if every weakly Cauchy sequence in $X$ is weakly convergent. Since $c_0$ is not weakly sequentially complete, it does not linearly embed into a~weakly sequentially complete space.

\begin{thm}[{\cite[Thm.~1.3]{CDW}}]\label{CDW_theorem}
For arbitrary $n\in\N$ and $M\subset\R^n$ the~Lipschitz-free space $\free{M}$ is weakly sequentially complete.
\end{thm}
\noindent
Note that in view of \cite[Cor.~3.3]{kaufmann}, this is equivalent to $\free{[0,1]^n}$ being weakly sequentially complete, where the~cube $[0,1]^n$ can be equipped with the~metric given by an~arbitrary norm on $\R^n$.

The~authors in \cite{CDW} next pose a~question whose negative answer could bring us closer to a~solution to the~original problem from \cite{DF}. Namely, they ask whether $c_0$ linearly embeds into $\free{\ell_2}$ (\cite[Question 3]{CDW}). We extend the~result from Theorem \ref{CDW_theorem} in the~spirit of the~proposed question. The~main result of the~present paper reads as follows.

\begin{mainthm}
If $M$ is a~compact subset of a~superreflexive Banach space, then the Lipschitz-free space $\free{M}$ is weakly sequentially complete.

\end{mainthm}

The~method of proving Theorem~\ref{CDW_theorem} in \cite{CDW} was based on a~direct application of Bourgain's result about the~weak sequential completeness of $(C^1([0,1]^n))^\ast$, the~dual of the~space of $C^1$-smooth functions on $[0,1]^n$, whereas our approach is based on adapting Bourgain's strategy and combining it with combinatorial properties of superreflexive spaces as well as certain approximation techniques for Lipschitz maps.

Note that Bourgain (\cite{bourgain1}, \cite{bourgain2}) actually proved something stronger than the~weak sequential completeness. Namely, he showed that a~certain class of subspaces of $C(S,E)$, where $S$ is a~compact Hausdorff space and $E$ is a~finite-dimensional Banach space, has the~so-called Pe\l czy\'nski's property ($V$)---a~condition introduced in \cite{pelczynski} which ultimately leads to the~weak sequential completeness of the~dual space (for details, see \cite[\S III.D]{W_book}). Recall that a~series $\sum_{n=1}^\infty x_n$ in a~Banach space $X$ is called {\it weakly unconditionally Cauchy} (WUC for short) if $\sum_{n=1}^\infty\abs{\ip{x_n}{x^\ast}}<\infty$ for every $x^\ast\in X^\ast$. A~Banach space $X$ is said to have {\it property} ($V$) provided that for every $K\subset X^\ast$, the~relative weak compactness of $K$ is equivalent to the~condition:
\begin{equation*}
\lim_{n\to\infty}\sup\bigl\{\abs{\ip{x_n}{x^\ast}}\colon x^\ast\in K\bigr\}=0\,\,\mbox{ for every WUC series }\sum_{n=1}^\infty x_n\mbox{ in }X,
\end{equation*}
and $X$ is said to have {\it property} ($V^\ast$) if for every $K\subset X$, the~relative weak compactness of $K$ is equivalent to the~condition:
\begin{equation*}
\lim_{n\to\infty}\sup\bigl\{\abs{\ip{x}{x_n^\ast}}\colon x\in K\bigr\}=0\,\,\mbox{ for every WUC series }\sum_{n=1}^\infty x_n^\ast\mbox{ in }X^\ast.
\end{equation*}
Among the results concerning these properties established in \cite{pelczynski}, we would like to point out  two relations crucial in our context. The first one says that if a Banach space $X$ has property ($V$), then $X^\ast$ has property ($V^\ast$), and by the other one, a Banach space $X$ with property ($V^\ast$) is weakly sequentially complete. Hence, the~key part of our argument is to show that for a compact subset $M$ of a superreflexive Banach space, the corresponding Lipschitz-free space $\free{M}$ admits property ($V^\ast$) (see Theorem~\ref{thm: technical theorem} below).

As defined by Godefroy and Talagrand in \cite{GT}, a Banach space $X$ has {\it property} ($X$) if for every $x^{\ast\ast}\in X^{\ast\ast}$, the condition
\begin{equation*}
\sum_{n=1}^\infty\ip{x_n^\ast}{x^{\ast\ast}}=\Big\langle w^\ast\mbox{-}\sum_{n=1}^\infty x_n^\ast,x^{\ast\ast}\Big\rangle\,\,\mbox{ for every WUC series }\sum_{n=1}^\infty x_n^\ast\mbox{ in }X^\ast
\end{equation*}
implies that $x^{\ast\ast}\in X$. Property ($X$) is strictly stronger than property ($V^\ast$), it is isomorphic invariant, and a Banach space with property ($X$) is strongly unique isometric predual of its dual (see \cite{GT}, \cite{HWW} and references therein). Weaver \cite{weaver_paper} recently showed that the Lipschitz-free space over a metric space with finite diameter or over a Banach space is the strongly unique predual of its dual. In \cite{GL}, Godefroy and Lerner formulate the following problem:
\begin{prob}[\cite{GL}]
\label{problem_X}
Let $n>1$. Does $\free{\R^n}$ have property ($X$)?
\end{prob}
\noindent
Theorem \ref{thm: technical theorem} provides thus also partial information in this line of investigation.

Let us mention that other nontrivial examples of metric spaces whose Lipschitz-free spaces are weakly sequentially complete, or even admit the~Schur property, include uniformly discrete metric spaces, snowflaking of any metric space (both to be found in \cite{kalton}), metric spaces that isometrically embed into an~$\R$-tree \cite{godard}, separable ultrametric spaces \cite{CD}, countable proper metric spaces (\cite{HLP}, \cite{petitjean} and \cite{dalet}), or metric spaces originating from $p$-Banach spaces with a~monotone FDD \cite{petitjean}.



\section{Preparations}
\noindent
For a~Banach space $X$ we denote by $B_X$ and $S_X$ the~unit ball and the~unit sphere of $X$, respectively.

Let $(M,d)$ be a~pointed metric space, that is, a~metric space with a~distinguished point $0\in M$. Then the~space $\Lipz{M}$ of all real-valued Lipschitz functions on $M$ which vanish at $0$, equipped with the~norm given by the~Lipschitz constant of a~function 
$$
\pn{f}{\Lipp}=\sup\set{\frac{\abs{f(p)-f(q)}}{d(p,q)}}{p,q\in M,\,p\neq q}\quad(f\in\Lipz{M}),
$$
is a~Banach space. The~metric space $M$ isometrically embeds into $\Lipz{M}^\ast$ via the~Dirac map $\map{\delta}{M}{\Lipz{M}^\ast}$ defined by
$\ip{f}{\delta(p)}=f(p)$ for $f\in\Lipz{M}$ and $p\in M$. The~\emph{Lipschitz-free space over M}, denoted $\free{M}$, is the~norm-closed linear span of\linebreak $\set{\delta(p)}{p\in M}$ in $\Lipz{M}^\ast$ with the~norm $\n{\,\cdot\,}_\FF$ induced by that of $\Lipz{M}^\ast$. Its dual space $\free{M}^\ast$ is linearly isometric to $\Lipz{M}$ and on the~unit ball of $\Lipz{M}$ the~weak$^\ast$ topology induced by $\free{M}$ coincides with the~topology of pointwise convergence.

Lipschitz-free spaces are characterized by their universality property which is illustrated by this diagram:
$$
\xymatrix{
M\ar[r]^L\ar[d]_{\delta} & X\\
\free{M}\ar[ur]_{\bar{L}} & 
}
$$
and reads as follows. If $M$ is a~pointed metric spaces, $X$ is a~Banach space and \mbox{$\map{L}{M}{X}$} is any Lipschitz map such that $L(0)=0$, then there exists a~unique linear map\linebreak \mbox{$\map{\bar L}{\free{M}}{X}$} such that $\bar L\circ\delta=L$ and $\|\bar L\|=\pn{L}{\Lipp}$ ({\it cf.} \cite[Lemma 3.2]{kalton}).  

For the~introduction to Lipschitz-free spaces (also known as Arens--Eells spaces) we refer the~reader to the~book \cite{weaver} by Weaver, fundamental papers \cite{GK} and \cite{kalton} by Godefroy--Kalton and Kalton, respectively, or the~latest survey \cite{G} by Godefroy.

In this section, we only recall facts that will later be used in our work. Let us begin by a~well-known observation, essential in the~theory of Lipschitz-free spaces, that any real-valued $L$-Lipschitz function $f$ on a~nonempty subset $N$ of a~metric space $(M,d)$ can be extended to an~$L$-Lipschitz function $\bar f$ on $M$. Indeed, apply for instance the~McShane's \cite{McShane} inf-convolution formula
$$
\bar f(p)=\inf\set{f(q)+Ld(p,q)}{q\in N}\quad (p\in M).
$$

One of the~key properties enjoyed by Lipschitz-free spaces, which has assured them an~important role in nonlinear functional analysis, is that they provide a~linearization of Lipschitz maps in the~following way. If we embed, through the~Dirac map $\delta$, pointed metric spaces $M$ and $N$ into the~corresponding Lipschitz-free Banach spaces $\free{M}$ and $\free{N}$, respectively, then any Lipschitz map $\map{L}{M}{N}$ such that $L(0)=0$ extends to a~bounded linear operator $\map{\hat L}{\free{M}}{\free{N}}$ with $\|\hat L\|=\pn{L}{\Lipp}$. That is, the~diagram below commutes:
$$
\xymatrix{
M\ar[r]^L\ar[d]_{\delta_M} & N\ar[d]^{\delta_N}\\
\free{M}\ar[r]^{\hat{L}} & \free{N}
}
$$
This follows easily from the~universality property when $\hat L=\overline{\delta_N\circ L}$. In fact, $\hat{L}$ is the~predual operator to $\map{L^{\#}}{\Lipz{N}}{\Lipz{M}}$ defined by $L^{\#}(F)=F\circ L$ (see \cite[Lemma~3.1]{kalton}). Consequently, if $M$ and $N$ are bilipschitz homeomorphic, then $\free{M}$ and $\free{N}$ are isomorphic; in particular, passing to a~strongly equivalent metric on a~metric space does not change the~isomorphism class of the~resulting Lipschitz-free space. Similarly, if $M$ is a~subspace of $N$, then $\free{M}$ is linearly isometric to a~subspace of $\free{N}$. 

The~approach in \cite{weaver} provides a~formula for the~norm on Lipschitz-free spaces which relies only on the~metric of the~underlying metric space and does not involve Lipschitz functions---a phenomenon referred to as Kantorovich--Rubinstein duality. To wit, we have
$$
\pn{\mu}{\FF}=\inf\set{\sum_{i=1}^k\abs{a_i}d(p_i,q_i)}{\mu=\sum_{i=1}^k a_i (\delta(p_i)-\delta(q_i))}
$$
for $\mu\in\aspan\set{\delta(p)}{p\in M}$ (where we adopt the~convention $\delta(0)=0$). A~detailed argument can be found, {\it e.g.}, in the~introduction to \cite{CDW}. The~above formula for $\n{\,\cdot\,}_\FF$ along with \cite[Lemma~3.100]{FHHMZ} yields that every $\mu\in\free{M}$ has a~representation
$$
\mu=\sum_{i=1}^{\infty}a_i\frac{\delta(p_i)-\delta(q_i)}{d(p_i,q_i)}
$$
with some $(a_i)_{i=1}^{\infty}\subset\ell_1$ and $(p_i,q_i)_{i=1}^{\infty}\subset\widetilde{M}:=\{(p,q)\in M^2\colon p\neq q\}$. Moreover, $\n{\mu}_\FF$ is the~infimum of the~$\ell_1$-norm of $(a_i)_{i=1}^\infty$ over all such representations.

Recall that a~Banach space $X$ is called {\it superreflexive} provided that every Banach space that is finitely representable in $X$ is reflexive; equivalently---every ultrapower of $X$ is reflexive. It is a~famous theorem by Enflo \cite{enflo} saying that $X$ is superreflexive if and only if it admits an~equivalent uniformly convex norm, that is, a~norm such that $\delta_X(\e)>0$ for each $\e\in (0,2]$, where $\delta_X(\e)$ is the~modulus of convexity of $X$. This happens to be also equivalent to admitting an~equivalent uniformly smooth norm, that is, a~norm for which the~modulus of smoothness $\varrho_X(\tau)=o(\tau)$ as $\tau\to 0$.

Pisier \cite{pisier}, using a~martingale-type approach, established a~precise quantitative version of Enflo's theorem. Namely, every superreflexive space can be renormed so that its modulus of convexity satisfies $\delta_X(\e)\geq c\e^q$ for each $\e\in (0,2]$ and some constants $c>0$ and $q\geq 2$; a~space admitting such a~renorming is called $q$-{\it convex}. Every superreflexive space can also be renormed so that the~modulus of smoothness satisfies $\varrho_X(\tau)\leq C\tau^p$ for each $\tau\in (0,\infty)$ and some constants $C>0$ and $1<p\leq 2$; a~space admitting such a~renorming is called $p$-{\it smooth}.

Among many permanence properties of superreflexive spaces we shall need the~following two: Firstly, $X$ is superreflexive if and only if $X^\ast$ is superreflexive---this follows immediately from the~well-known duality formula  
$$
\varrho_X(\tau)=\sup\Bigl\{\frac{\tau\e}{2}-\delta_{X^\ast}(\e)\colon 0<\e\leq 2\Bigr\}\qquad (\tau>0)
$$
(see \cite[Lemma~9.8]{FHHMZ}) which shows that $X^\ast$ is $q$-convex if and only if $X$ is $p$-smooth, where $1<p\leq 2\leq q<\infty$ satisfy $p^{-1}+q^{-1}=1$. Secondly, if $X$ is superreflexive, then the~Lebesgue--Bochner space $L_2(X)$ of all $X$-valued Bochner square integrable functions on $[0,1]$ is superreflexive too. Moreover, there is a~precise quantitative statement of this fact due to Figiel \cite{figiel} and Figiel and Pisier \cite{figiel_pisier}.
\begin{thm}[{\cite[Thm.~1.e.9]{LT}}]\label{thm: Figiel-Pisier}
For every Banach space $X$ there exist constants $a,b>0$ such that
$$
\delta_X(\e)\geq\delta_{L_2(X)}(\e)\geq a\delta_X(b\e)\quad\mbox{for every }\e\in (0,2].
$$
\end{thm}
\noindent
Consequently, for any $q\geq 2$, $X$ is $q$-convex if and only if so is $L_2(X)$, which in turn implies that $L_2(X^\ast)$ is $q$-convex whenever $X$ is $p$-smooth and $p$, $q$ are conjugate exponents. Note also that the~same conclusions about behavior of the~modulus of convexity hold true for any $\ell_2$-sum of finitely many copies of $X$, as it linearly and isometrically embeds into $L_2(X)$. We will make use of these observations in the~proof of the~Main Theorem.

Several characterizations of superreflexive spaces in terms of certain combinatorial properties of norm were given by James. One of them (\cite{james1}, \cite{james2}) states that $X$ is superreflexive if and only if given any $\e>0$ there is $n\in\N$ such that for any vectors $x_1,\ldots ,x_n\in B_X$ there exists $1\leq k\leq n$ and $y\in\conv\{x_1,\ldots,x_k\}$, $z\in\conv\{x_{k+1},\ldots,x_n\}$ with $\n{y-z}<\e$. In the~next section, we provide a~strengthening of this condition which is based on a~certain Clarkson-type inequality for two equivalent norms on a~superreflexive space (see Lemma~\ref{lem: difference of averages}). For more information on superreflexive spaces, see {\it e.g.} \cite[Ch.~9]{FHHMZ} and the~references therein.


We shall also need two facts about approximation. The~first one is a~deep theorem of H\'ajek and Johanis on approximation of Lipschitz functions by smooth functions on sufficiently smooth Banach spaces. By $C^k(X)$ we denote the~space of all $k$-times continuously Fr\'echet differentiable functions on a~Banach space $X$. Recall that a~function $\varphi\colon X\to\R$ is called a~{\it bump function} if the~set $\{x\in X\colon\varphi(x)\not=0\}$ is nonempty and bounded. The~already mentioned renorming theorems for superreflexive spaces imply that every such space admits an~equivalent Fr\'echet differentiable norm (see {\it e.g.} \cite[Thm.~9.14]{FHHMZ}) and hence it admits a~Lipschitz $C^1$-smooth bump function (see \cite[Fact~I.2.1]{DGZ}). Therefore, the~following theorem applies to all superreflexive Banach spaces with $k=1$.

\begin{thm}[{\cite[Cor.~8]{HJ}}]
\label{thm: approximation by differentiable functions}
	Let $X$ be a~separable normed space that admits a~$C^k$-smooth Lipschitz bump function, for some $k\in\N\cup\{0\}$. There exists a~constant $K\geq 1$ depending only on $X$ such that for every $L$-Lipschitz function $f\colon X\to\R$ and any $\e>0$ there exists a~$KL$-Lipschitz function $g\in C^k(X)$ such that $\sup_{x\in X}\abs{f(x)-g(x)}<\e$.
\end{thm}

It is worth mentioning that Theorem~\ref{thm: approximation by differentiable functions} was preceded by a~result of Cepedello-Boiso quoted below. However, the~crucial advantage of the~H\'ajek--Johanis theorem lies in the~fact that it gives a~control on the~Lipschitz constant of the~approximating function $g$ which will be of great importance in the~proof of the~Main Theorem.
\begin{thm}[{\cite[Cor.~3]{cepedello}}]
Let $X$ be a~superreflexive Banach space and let $\alpha\in (0,1]$ be such that $X$ is $(1+\alpha)$-smooth. Then for every Lipschitz function $f\colon X\to\R$ and any $\e>0$ there exists a~Fr\'echet differentiable map $g\colon X\to\R$ with its derivative $\alpha$-H\"older on bounded sets and such that $\sup_{x\in X}\abs{f(x)-g(x)}<\e$.
\end{thm}

The~second tool of approximation theory that we need is a~rather easy lemma on approximating uniformly continuous functions by Lipschitz ones.
\begin{lem}[{\cite[Ch.~7, Lemma~40]{HJ_book}}]
\label{lem: approximation by Lipschitz functions}
Let $(M,d)$ be a~metric space and let $f\colon M\to\R$ be a~uniformly continuous function with modulus of continuity
$$
\omega_f(t)=\sup\bigl\{\abs{f(x)-f(y)}\colon x,y\in M,\, d(x,y)\leq t\bigr\}\quad (t\geq 0).
$$
Assume that $\omega\colon [0,\infty)\to [0,\infty)$ is a~subadditive modulus of $f$, i.e. $\omega$ is nondecreasing, continuous at zero, $\omega(0)=0$, $\omega_f\leq\omega$ and $\omega(t+u)\leq\omega(t)+\omega(u)$ for all $t,u\in [0,\infty)$. Given any $\e, a>0$ with $\omega(a)\leq\e$, there exists an~$\frac{\e}{a}$-Lipschitz function $g\colon M\to\R$ such that $\sup_{x\in M}\abs{f(x)-g(x)}<\e$.
\end{lem}

Note that if $M$ above is a~convex subset of a~normed linear space (the~situation to which Lemma~\ref{lem: approximation by Lipschitz functions} will be applied), then the~minimal modulus $\omega_f$ of $f$ is subadditive.

\section{Proof of the Main Theorem}
\noindent
We start by the~announced lemma which strengthens James' characterization of superreflexivity and generalizes the~observation based on geometry of Hilbert spaces used in the~original Bourgain's result ({\it cf. }the~proof of \cite[Lemma 2]{bourgain1} and \cite[Lemma~III.D.32]{W_book}).
\begin{lem}
\label{lem: difference of averages}
	Let $(X,\ndot)$ be a~$q$-convex Banach space, $q\in[2,\infty)$, and let $n\in\N$ and $x_1,\ldots,x_n\in B_X$. Then there exist nonempty sets $A,B\subset\{1,\ldots,n\}$ with $\max A<\min B$ such that
	$$
	\Biggl\|\frac{1}{\abs{A}}\sum_{i\in A}x_i-\frac{1}{\abs{B}}\sum_{j\in B}x_j\Biggr\|\leq\frac{\gamma}{(\log_2 n)^{1/q}},
	$$
	where $\gamma\geq 1$ is a~constant depending only on $X$.
\end{lem}
\begin{proof}
	By Pisier's results \cite[Prop.~2.4, Thm~3.1]{pisier}, there exist a~constant $C\geq 1$ and a~norm $\tndot$ on $X$ such that $\n{x}\leq\tn{x}\leq C\n{x}$ for $x\in X$ and that
	\begin{equation}\label{L1}
		\tnbig{\frac{x+y}{2}}^q\leq\frac{\tn{x}^q+\tn{y}^q}{2}-\Biggl\|\frac{x-y}{2}\Biggr\|^q\quad\mbox{for all }x,y\in X.
	\end{equation}
	For any given $x_1,\ldots,x_n\in B_X$ define
	$$
	\alpha_k=\sup\Biggl\{\tnbig{\frac{1}{\abs{A}}\sum_{i\in A}x_i}\colon\abs{A}=2^k,\, A\subset\{1,\ldots,n\}\Biggr\}
	$$
	and
	$$
	\beta_k=\inf\Biggl\{\Biggl\|\frac{1}{\abs{A}}\sum_{i\in A}x_i-\frac{1}{\abs{B}}\sum_{j\in B}x_j\Biggr\|\colon \abs{A}=\abs{B}=2^k,\, \max A<\min B,\, A,B\subset\{1,\ldots,n\}\Biggr\}
	$$
	for $k=0,1,\ldots,\lfloor\log_2n\rfloor-1$. By inequality \eqref{L1} we have 
	$$
	\max\bigl\{\tn{x}^q,\tn{y}^q\bigr\}\geq\tnbig{\frac{x+y}{2}}^q+\Biggl\|\frac{x-y}{2}\Biggr\|^q,
	$$
	hence $\alpha_k^q\geq \alpha_{k+1}^q+2^{-q}\beta_k^q$ for each $k=0,1,\ldots,\lfloor\log_2n\rfloor-2$. Therefore,
	$$
	\alpha_0^q\geq\alpha_1^q+2^{-q}\beta_0^q\geq\ldots\geq \alpha_{\lfloor\log_2n\rfloor-1}^q+2^{-q}\!\!\sum_{k=0}^{\lfloor\log_2n\rfloor-2}\beta_k^q
	$$
	and, since $\alpha_0\leq C$ and $\alpha_{\lfloor\log_2n\rfloor-1}\geq\frac{1}{2}\beta_{\lfloor\log_2n\rfloor-1}$, we obtain
	$$
	\sum_{k=0}^{\lfloor\log_2n\rfloor-1}\beta_k^q\leq C^q\cdot 2^q.
	$$
	Consequently, there must exist $0\leq k<\lfloor\log_2n\rfloor$ such that 
	$$
	\beta_k^q\leq\frac{C^q\cdot 2^q}{\lfloor\log_2n\rfloor}
	$$
	and the~assertion follows.
\end{proof}
\begin{rem}\label{rem: Pisier's constant}
An inspection of Pisier's proof of \cite[Thm.~3.1]{pisier} shows that the~constant $C\geq 1$ (and hence also the~resulting constant $\gamma$) depends only on the~behavior of the~modulus of convexity of $X$, more precisely, on the~constants $c>0$ and $q\geq 2$ for which we have $\delta_X(\e)\geq c\e^q$ for $\e\in (0,2]$ (see the~proof of \cite[Prop.~2.4]{pisier}).
\end{rem}

In the~proof of Theorem \ref{thm: technical theorem} below, we shall use the~following simple observation (see also \cite[pp.~170--171]{W_book}): If we divide the~set 
$
\set{(n,j)}{n\in\N, 1\leq j\leq n}
$
into finitely many subsets, then at least one of them must contain an~infinite subset
$
\set{(n_i,j_i^s)}{i\in\N, 1\leq s\leq i}
$
such that $(n_i)_{i=1}^\infty$ is a~strictly increasing sequence in $\N$ and $1\leq j_i^1<\dots<j_i^i\leq n_i$ for each $i\in\N$. 

\begin{thm}
\label{thm: technical theorem}
Let $X$ be a~superreflexive Banach space and let $M\subset X$ be a~compact set with $0\in M$. If $\Gamma\subset\free{M}$ is bounded and not relatively weakly compact, then there exists a~WUC series $\sum_{k=1}^\infty \varrho_k$ in $\Lipz{M}$ such that 
$$\limsup_{k\to\infty}\,\sup\bigl\{\abs{\ip{\mu}{\varrho_k}}\colon \mu\in \Gamma\bigr\}>0.$$
\end{thm}

\begin{proof}
As we have already noted, passing to an~equivalent norm of $X$ does not change the~isomorphism class of $\free{M}$. Therefore, by Pisier's theorem we may (and we do) assume that $X$ is $p$-smooth with some $p\in (1,2]$. Since $M$ is compact, we can also assume $X$ to be separable.

Following the~proof of the~Eberlein--\v Smulyan theorem (see {\it e.g.} \cite[\S II.C]{W_book}), we find constants $C\geq 1$ and $\xi>0$ and, for each $n\in\N$, sequences $$(f_1^n,\dots,f_n^n)\subset\Lipz{M}$$
and
$$(\mu_1^n,\dots,\mu_n^n)\subset\aspan\{\delta(p)\colon p\in M\}$$
satisfying the~following conditions:
\begin{align}
	\label{norm of mu}
\n{\mu_j^n}_{\FF}\leq C\,&\,\,\,\,\,\, \textup{ for }\, n\in\N\, \textup{ and }\, 1\leq j\leq n,\\
\label{norm of f}
\n{f_j^n}_{\Lipp}\leq 1\,&\,\,\,\,\,\, \textup{ for }\, n\in\N\, \textup{ and }\, 1\leq j\leq n,\\
\label{small to the left}
\left|\ip{ \mu_j^n}{f_k^n}\right|\leq\frac\xi 3\,& \,\,\,\,\,\,\textup{ for }\, j,k,n\in\N \,\textup{ with }\, 1\leq j<k\leq n,\\
\label{big to the right}
\ip{\mu_j^n}{f_k^n}\geq\xi\,&\,\,\,\,\,\, \textup{ for }\, j,k,n\in\N\,\textup{ with }\, 1\leq k\leq j\leq n,\\
\label{mu close to gamma}
\dist\left(\mu_j^n,\Gamma\right)\leq\frac {6C\xi}{\xi+48C}\,&\,\,\,\,\,\, \textup{ for }\, n\in\N\, \textup{ and }\, 1\leq j\leq n.
\end{align}
Due to the~aforementioned existence of norm-preserving extensions of Lipschitz functions, we may assume that all $f_k^n$'s are actually defined on $rB_X$ for some $r>0$ satisfying \mbox{$rB_{X}\supset M$.} Then we regard $\mu_k^n$'s as elements of $\free{rB_{X}}$ with supports in $M$. Next, by Theorem~\ref{thm: approximation by differentiable functions}, for each $n\in\N$ and $1\leq k\leq n$ there is a~$K$-Lipschitz $C^1$-smooth function which uniformly approximates $f_k^n$, where $K\geq 1$ depends only on $X$. So, after normalizing and adjusting $\xi$, we may moreover assume that $f_k^n$'s are continuously Fr\' echet differentiable. For future reference, when exact indices will not be clear or important, we denote $\Omega=\{\mu_j^n\colon n\in\N, 1\leq j\leq n\}$.

Write $\Lip_{0,F}\left(rB_{X}\right)$ for the~subspace of $\Lipz{rB_{X}}$ consisting of all continuously Fr\' echet differentiable functions and let $C_b\left(rB_X,X^\ast\right)$ be the~space of bounded continuous maps from $rB_{X}$ to $X^\ast$ equipped with the~norm 
$$
\|F\|_\infty=\sup\{\|F(p)\|_{X^\ast}\colon p\in rB_{X}\}\quad \left(F\in C_b\left(rB_X,X^\ast\right)\right).
$$
Let $\map{\Phi}{\Lip_{0,F}\left(rB_{X}\right)}{C_b\left(rB_X,X^\ast\right)}$ be the~Fr\'echet derivative map, that is,
$$
\Phi(f)(p)=\mathrm{d}f(p)\,\,\,\textup{ for }\,f\in\Lip_{0,F}\left(rB_{X}\right)\,\textup{ and }\,p\in rB_{X}.
$$
Then $\Phi$ is a~linear isometry because by the~mean value theorem (see \cite[Ch.~1, Prop.~65]{HJ_book}) we have $\n{f}_\Lipp\leq\n{\Phi f}_\infty$ for each $f\in\Lip_{0,F}(rB_X)$, whereas the~converse inequality is obvious by the~definition of the~Fr\'echet derivative. Denote by $\mathfrak{X}\subset C_b(rB_X,X^\ast)$ the~range of $\Phi$.

If we express the~norm on $\free{rB_{X}}$ via the~`metric formula' stated in the~previous section, for each $n\in\N$ and $1\leq j\leq n$ we can find $I_{n,j}\in\N$ and sequences
$$
\bigl(a^{(n,j)}_i\bigr)_{i=1}^{I_{n,j}}\subset \R\quad\mbox{ and }\quad\bigl(p^{(n,j)}_i,q^{(n,j)}_i\bigr)_{i=1}^{I_{n,j}}\subset \widetilde{M}
$$
such that 
\begin{equation}
\label{eq: representation of mu}
\mu^n_j=\sum_{i=1}^{I_{n,j}} a^{(n,j)}_i\bigl(\delta(p^{(n,j)}_i)-\delta(q^{(n,j)}_i)\bigr)
\end{equation}
and
\begin{equation}
\label{norm of extension}
\sum_{i=1}^{I_{n,j}} \bigl|a^{(n,j)}_i\bigr|\bigl\|p^{(n,j)}_i-q^{(n,j)}_i\bigr\|\leq 2\n{\mu_j^n}_{\FF}.
\end{equation}
Notice that for every $F\in\mathfrak{X}$ we have
$$
\ip{F}{(\Phi^{-1})^\ast\mu^n_j}=\ip{\mu_j^n}{\Phi^{-1}F}=\sum_{i=1}^{I_{n,j}}a_i^{(n,j)}\bigl(\Phi^{-1}F(p_i^{(n,j)})-\Phi^{-1}F(q_i^{(n,j)})\bigr),
$$
thus, by the~Newton--Leibniz formula, we infer that the~functional $\nu_j^n\in (C_b(rB_X,X^\ast))^\ast$ defined by 
\begin{equation}
\label{definition of nu}
\ip{F}{\nu_j^n}=\sum_{i=1}^{I_{n,j}} a^{(n,j)}_i\lint{0}{1}{\ip{p^{(n,j)}_i-q^{(n,j)}_i}{F\bigl(q^{(n,j)}_i+t\bigl(p^{(n,j)}_i-q^{(n,j)}_i\bigr)\bigr)}}{t} 
\end{equation}
is an~extension of $(\Phi^{-1})^\ast(\mu_j^n)\in\mathfrak{X}^\ast$. Moreover, inequality \eqref{norm of extension} implies that 
\begin{align}
\label{norm of nu depends on convM}
\nonumber
\left|\ip{F}{\nu_j^n}\right|&\leq\sum_{i=1}^{I_{n,j}} \bigl|a^{(n,j)}_i\bigr|\lint{0}{1}{\bigl\| F\bigl(q^{(n,j)}_i+t\bigl(p^{(n,j)}_i-q^{(n,j)}_i\bigr)\bigr)\bigr\|_{X^\ast}\bigl\|p^{(n,j)}_i-q^{(n,j)}_i\bigr\|}{t}\\
&\leq 2\|\mu_j^n\|_{\FF}\sup_{p\in\overline{\conv}M}\|F(p)\|_{X^\ast},
\end{align}
whence $\|\nu_j^n\|\leq 2\n{\mu_j^n}_{\FF}$.

For any pair $(n,j)$ with $n\in\N$ and $1\leq j\leq n$ we consider the~Banach space
$$
\mathfrak{Z}_{n,j}=\Biggl(\bigoplus_{i=1}^{I_{n,j}}L_2([0,1],X^\ast)\Biggr)_{\!\!\ell_2}.
$$
In view of the~remarks following Theorem~\ref{thm: Figiel-Pisier}, the~$p$-smoothness of $X$ yields that every such space is $q$-convex with $q\in [2,\infty)$ being the~conjugate exponent to $p$. Moreover, Theorem~\ref{thm: Figiel-Pisier} implies that $\delta_{\mathfrak{Z}_{n,j}}(\e)\geq c\e^q$ for each $\e\in (0,2]$ and with a~constant $c>0$ common for all $(n,j)$'s. 

Fix a~sequence $(\varepsilon_k)_{k=1}^\infty\subset(0,1)$ such that 
\begin{equation}
\label{eq:sum epsilon}
\sum_{k=1}^\infty\varepsilon_k<\frac{\xi}{144C}.
\end{equation}
Find $M_1\in\N$ so large that 
$$
\frac{\gamma\sqrt{2C}}{(\log_2M_1)^{1/q}}<\varepsilon_1,
$$
where $C$ comes from (\ref{norm of mu}) and $\gamma$ is the~constant produced by Lemma~\ref{lem: difference of averages} applied to any of the~spaces $\mathfrak{Z}_{n,j}$ (notice that Remark~\ref{rem: Pisier's constant} guarantees that the~same value of $\gamma$ works for all pairs $(n,j)$). For any pair $(n,j)$ with $n>M_1$ and $1\leq j\leq n$, and each $1\leq l\leq M_1$ we set
$$
u_l^{(n,j)}=\left(\sqrt{\bigl|a_i^{(n,j)}\bigr|\bigl\|p^{(n,j)}_i-q^{(n,j)}_i\bigr\|}\Phi(f_l^{M_1})\circ\gamma_i^{(n,j)}\right)_{i=1}^{I_{n,j}}\in\mathfrak{Z}_{n,j},
$$
where $a_i^{(n,j)}$, $p^{(n,j)}_i$, $q^{(n,j)}_i$ come from (\ref{eq: representation of mu}) and $\map{\gamma_i^{(n,j)}}{[0,1]}{X}$ is defined by
$$
\gamma_i^{(n,j)}(t)=q^{(n,j)}_i+t\bigl(p^{(n,j)}_i-q^{(n,j)}_i\bigr)\qquad (t\in [0,1]).
$$
Plainly, by \eqref{norm of f}, \eqref{norm of extension} and \eqref{norm of mu}, we have 
$$
\|u_l^{(n,j)}\|_{\mathfrak{Z}_{n,j}}\leq\sqrt{2C}.
$$
Thus, by Lemma \ref{lem: difference of averages}, there exist subsets $A_{n,j}, B_{n,j}$ of $\{1,\dots,M_1\}$ with $\max A_{n,j}<\min B_{n,j}$ such that
\begin{equation}
\label{eq: application of Lemma 1}
\Biggl\|\frac{1}{|A_{n,j}|}\sum_{l\in A_{n,j}}u_l^{(n,j)}-\frac{1}{|B_{n,j}|}\sum_{l\in B_{n,j}}u_l^{(n,j)}\Biggr\|_{\mathfrak{Z}_{n,j}}\!\!<\varepsilon_1.
\end{equation}
Since there are only finitely many subsets of $\{1,\dots,M_1\}$, we can find $A, B\subset\{1,\dots,M_1\}$ and an~infinite set $\{(n_i,j_i^s)\colon i\in\N,\, 1\leq s\leq M_1+i\}$, where:
\begin{itemize}
\item $n_1>M_1$,

\item $(n_i)_{i=1}^\infty\subset\N$ is strictly increasing,

\item $1\leq j_i^1<\dots<j_i^{M_1+i}\leq n_i$ for each $i\in\N$,
\end{itemize}
such that
$$
A_{n_i,j_i^s}=A\quad\mbox{ and }\quad B_{n_i,j_i^s}=B\quad\mbox{ for all }i\in\N\mbox{ and }1\leq s\leq M_1+i 
$$
(see the~remark above the~statement of Theorem \ref{thm: technical theorem}). Of course, the~sequences $$
\bigl(f_{j_i^1}^{n_i},\dots,f_{j_i^{M_1+i}}^{n_i}\bigr)\subset\Lipf\,\,\mbox{ and }\,\,\bigl(\mu_{j_i^1}^{n_i},\dots,\mu_{j_i^{M_1+i}}^{n_i}\bigr)\subset\aspan\{\delta(p)\colon p\in M\}
$$
for $i\in\N$ satisfy conditions (\ref{norm of mu})--(\ref{big to the right}) with obvious substitution of indices. Therefore, we relabel these sequences as $(f_1^n,\dots,f_n^n)$ and $(\mu_1^n,\dots,\mu_n^n)$, where $n\in\{M_1+1,M_1+2,\ldots\}$. Similarly, $(\nu_1^n,\dots,\nu_n^n)\subset (C_b(rB_X,X^\ast))^\ast$ are the~corresponding extended functionals as in (\ref{definition of nu}). Further, we define $\varphi_1$ as the~constant $1$ function on $rB_X$,
$$
z_1=\frac{1}{2|A|}\sum_{l\in A}f_l^{M_1}-\frac{1}{2|B|}\sum_{l\in B}f_l^{M_1}\in\Lipf,
$$
and
$$
\kappa_1=\mu_{\max A}^{M_1}\in\aspan\set{\delta(p)}{p\in M},\quad \lambda_1=\nu_{\max A}^{M_1}\in (C_b(rB_X,X^\ast))^\ast.
$$
Then $\|\Phi(z_1)\|_\infty \leq 1$ as $z_1$ obviously lies in the~unit ball by (\ref{norm of f}), and, in view of inequalities (\ref{small to the left}) and (\ref{big to the right}), we have 
$$
\left|\ip{\kappa_1}{z_1}\right|\geq\frac {\xi}{3}.
$$
Moreover, combining H\" older's inequality with (\ref{norm of extension}) and (\ref{eq: application of Lemma 1}), we obtain
\begin{equation}\label{eq: combining Holder}
\sum_{i=1}^{I_{n,j}} \bigl|a^{(n,j)}_i\bigr|\bigl\|p^{(n,j)}_i-q^{(n,j)}_i\bigr\|\lint{0}{1}{\bigl\| \Phi(z_1)\bigl(q^{(n,j)}_i+t\bigl(p^{(n,j)}_i-q^{(n,j)}_i\bigr)\bigr)\bigr\|_{X^\ast}}{t}<\sqrt{2C}\varepsilon_1
\end{equation}
for every pair $(n,j)$ with $n\in\{M_1+1,M_1+2,\dots\}$ and $1\leq j\leq n$.
 
Now, to proceed with inductive construction, fix any $k\in\N$, $k\geq 2$ and assume that we have already defined:
\begin{itemize}
\item natural numbers $M_1<\ldots<M_{k-1}$,
\item $(f_1^n,\dots,f_n^n)\subset\Lipf$ for $n>M_{k-1}$,
\item $(\mu_1^n,\dots,\mu_n^n)\subset\aspan\set{\delta(p)}{p\in M}$ for $n>M_{k-1}$,
\item $(z_1,\dots,z_{k-1})\subset\Lipf$,
\item $(\kappa_1,\dots,\kappa_{k-1})\subset\Omega$ (to recall the~definition of $\Omega$ see the~beginning of the~proof) and
\item $(\varphi_1,\dots,\varphi_{k-1})$, a~sequence of $C^1$-smooth Lipschitz real-valued functions on $rB_X$,
\end{itemize}
such that:
\begin{enumerate}[label=(\roman*)]
\setlength{\itemsep}{3pt}

\item the~sequences $(f_1^n,\dots,f_n^n)$ and $(\mu_1^n,\dots,\mu_n^n)$, for $n>M_{k-1}$, are relabeled copies of some of the~original $f_j^n$'s and $\mu_j^n$'s which still satisfy conditions \eqref{norm of mu}--\eqref{big to the right};
\item $\n{\Phi(z_l)}_\infty\leq 1$ for each $1\leq l\leq k-1$;
\label{norm of z_k}
\item $\abs{\ip{\kappa_l}{z_l}}\geq\xi/3$ for each $1\leq l\leq k-1$;
\label{kappa_k at z_k}
\item $\sup\limits_{p\in\overline{\conv}M}\Bigl\|\Phi(\varphi_{l} z_l)(p)-\prod\limits_{i=1}^{l-1}\left(1-\n{\Phi(z_i)(p)}_{X^\ast}\right)\Phi(z_l)(p)\Bigr\|_{X^\ast}\!\!<3\e_l$ for each $1\leq l\leq k-1$;
\label{close element}
\item for each $1\leq l\leq k-1$, inequality \eqref{eq: combining Holder} holds true for every pair $(n,j)$ with $n>M_l$ and $1\leq j\leq n$, and with the~right-hand side replaced by $\sqrt{2C}\e_l$.
\label{complicated inequality}

\end{enumerate}

Since the~derivatives of $f_j^n$'s are continuous, the~function 
$$
\Psi_k:=\map{\prod\limits_{i=1}^{k-1}\left(1-\n{\,\cdot\,}_{X^\ast}\circ\Phi(z_i)\right)}{rB_{X}}{[0,1]}
$$
is uniformly continuous on each compact subset of its domain. Therefore, Lemma~\ref{lem: approximation by Lipschitz functions} produces a~Lipschitz function which uniformly approximates $\Psi_k$ on the~compact set $\overline{\conv}M$. Now, an~appeal to Theorem \ref{thm: approximation by differentiable functions} gives a~$C^1$-smooth  Lipschitz function $\map{\varphi_k}{rB_{X}}{\R}$ such that 
\begin{equation}
\label{eq: approximation by phi}
\abs{\Psi_k(p)-\varphi_k(p)}<\e_k\quad\mbox{for each }p\in\overline\conv\,M.
\end{equation}
Set $D_k=\|\varphi_k\|_{\Lipp}$ and find a~finite $\e_k/D_k$-dense subset $S_k\ni 0$ of $\overline{\conv}\,M$. Pick also a~natural number $M_k>M_{k-1}$ so large that
$$
\frac{\gamma\sqrt{2C+\abs{S_k}(D_kr)^2}}{(\log_2M_k)^{1/q}}<\varepsilon_k.
$$

For any pair $(n,j)$ with $n>M_k$ and $1\leq j\leq n$, and for each $1\leq l\leq M_k$ we define
$$
u_l^{(n,j)}=\left(\sqrt{\bigl|a_i^{(n,j)}\bigr|\bigl\|p^{(n,j)}_i-q^{(n,j)}_i\bigr\|}\Phi(f_l^{M_k})\circ\gamma_i^{(n,j)}\right)_{i=1}^{I_{n,j}}\in \mathfrak{Z}_{n,j},
$$
$$
v_l=\bigl(D_kf_l^{M_k}(p)\bigr)_{p\in S_k}\in \ell_2^{\abs{S_k}}
$$
and
$$
w_l^{(n,j)}=\bigl(u_l^{(n,j)},v_l\bigr)\in\mathfrak{Z}_{n,j}\oplus\ell_2^{|S_k|}.
$$
Then, by inequalities \eqref{norm of f}, \eqref{norm of extension} and \eqref{norm of mu}, we have
$$
\bigl\|w_l^{(n,j)}\bigr\|_{\mathfrak{Z}_{n,j}\oplus\ell_2^{|S_k|}}\leq\sqrt{2C+|S_k|(D_kr)^2}.
$$
Hence, from Lemma \ref{lem: difference of averages} it follows that there exist subsets $A_{n,j}, B_{n,j}$ of $\{1,\dots,M_k\}$ with $\max A_{n,j}<\min B_{n,j}$ such that
\begin{equation}
\label{eq:application of Lemma k}
\Biggl\|\frac{1}{|A_{n,j}|}\sum_{l\in A_{n,j}}w_l^{(n,j)}-\frac{1}{|B_{n,j}|}\sum_{l\in B_{n,j}}w_l^{(n,j)}\Biggr\|_{\mathfrak{Z}_{n,j}\oplus\ell_2^{|S_k|}}<\varepsilon_k.
\end{equation}

As before, since there are only finitely many subsets of $\{1,\dots,M_k\}$, we can find subsets $A, B$ of $\{1,\dots,M_k\}$ and an~infinite set $\{(n_i,j_i^s)\colon i\in\N,\, 1\leq s\leq M_k+i\}$, where:
\begin{itemize}
\item $n_1>M_k$,

\item $(n_i)_{i=1}^\infty\subset\N$ is strictly increasing,

\item $1\leq j_i^1<\dots<j_i^{M_k+i}\leq n_i$ for each $i\in\N$,
\end{itemize}
such that
$$
A_{n_i,j_i^s}=A\quad\mbox{ and }\quad B_{n_i,j_i^s}=B\quad\mbox{ for all }i\in\N\mbox{ and }1\leq s\leq M_k+i.
$$
Again, we relabel the~sequences 
$$
\bigl(f_{j_i^1}^{n_i},\dots,f_{j_i^{M_k+i}}^{n_i}\bigr)\subset\Lipf\,\,\mbox{and }\,\, \bigl(\mu_{j_i^1}^{n_i},\dots,\mu_{j_i^{M_k+i}}^{n_i}\bigr)\subset\aspan\set{\delta(p)}{p\in M}\quad (i\in\N)
$$
as $(f_1^n,\ldots,f_n^n)$ and $(\mu_1^n,\ldots,\mu_n^n)$, respectively, where $n\in\{M_k+1,M_k+2,\ldots\}$. As previously, $(\nu_1^n,\dots,\nu_n^n)\subset (C_b(rB_X,X^\ast))^\ast$ are the~corresponding extensions of $(\Phi^{-1})^\ast(\mu^n_j)$'s.

Define
$$
z_k=\frac{1}{2|A|}\sum_{l\in A}f_l^{M_k}-\frac{1}{2|B|}\sum_{l\in B}f_l^{M_k}\in\Lipf,
$$
and
$$
\kappa_k=\mu_{\max A}^{M_k}\in\aspan\set{\delta(p)}{p\in M},\quad \lambda_k=\nu_{\max A}^{M_k}\in (C_b(rB_X,X^\ast))^\ast.
$$
Then, plainly we have $\|\Phi(z_k)\|_{\infty}\leq 1$
and, by (\ref{small to the left}) and (\ref{big to the right}), also $\left|\ip{\kappa_k}{z_k}\right|\geq\xi/3$.

Moreover, (\ref{eq:application of Lemma k}) yields that 
$|D_kz_k(p)|<{\varepsilon_k}$ for all $p\in S_k$. This means that $D_kz_k$ is small on the~whole $\overline{\conv}M$. Indeed, $D_kz_k$ is a~$D_k$-Lipschitz function and for every $q\in \overline{\conv}M$ we can find $p\in S_k$ such that $\|p-q\|\leq\varepsilon_k/D_k$. Hence
\begin{equation}
\label{eq: z_k small}
|D_kz_k(q)|<2\varepsilon_k\,\,\,\,\,\, \textup{ for all }q\in \overline{\conv}M.
\end{equation} 
Since both $\varphi_k$ and $z_k$ are bounded, Lipschitz and differentiable with continuous derivatives on $rB_X$, and since $z_k(0)=0$, we have that $\varphi_k z_k\in\Lipf$ and 
$$
\Phi(\varphi_k z_k)=\Phi(\varphi_k)z_k+\varphi_k \Phi(z_k).
$$
Hence, for $p\in\overline{\conv}M$,
\begin{align*}
\|\Phi(\varphi_k z_k)(p)-\Psi_k(p)\Phi(z_k)(p)\|_{X^\ast}\leq\, &\abs{z_k(p)}\|\Phi(\varphi_k)(p)\|_{X^\ast}\\
&+\left|\Psi_k(p)-\varphi_k(p)\right|\|\Phi(z_k)(p)\|_{X^\ast}<3\varepsilon_k,
\end{align*}
where the~last inequality follows from (\ref{eq: z_k small}) and (\ref{eq: approximation by phi}).

Observe also that H\"older's inequality, jointly with (\ref{norm of extension}) and (\ref{eq:application of Lemma k}) gives
\begin{equation*}
\sum_{i=1}^{I_{n,j}} \bigl|a^{(n,j)}_i\bigr|\bigl\|p^{(n,j)}_i-q^{(n,j)}_i\bigr\|\lint{0}{1}{\bigl\| \Phi(z_k)\bigl(q^{(n,j)}_i+t\bigl(p^{(n,j)}_i-q^{(n,j)}_i\bigr)\bigr)\bigr\|_{X^\ast}}{t}<\sqrt{2C}\varepsilon_k,
\end{equation*}
for each pair $(n,j)$ with $n\in\{M_k+1,M_k+2,\dots\}$ and $1\leq j\leq n$. 

Therefore, all the~conditions (i)--(v) are satisfied with $k$ in the~place of $k-1$ and hence our inductive construction is complete.

Now, we shall show that the~series $\sum_{k=1}^{\infty}\varphi_kz_k\res{M}$ in $\Lipz{M}$ is WUC. Using conditions \ref{close element} and \ref{norm of z_k}, along with definition (\ref{eq:sum epsilon}), for every $p\in\overline{\conv}M$ we obtain
\begin{align*}
\sum_{k=1}^{\infty}\left\|\Phi(\varphi_kz_k)(p)\right\|_{X^\ast}&\leq\sum_{k=1}^{\infty}3\varepsilon_k+\sum_{k=1}^{\infty}\Psi_k(p)\left\|\Phi(z_k)(p)\right\|_{X^\ast}\\
&=\sum_{k=1}^{\infty}3\varepsilon_k+\sum_{k=1}^{\infty}\prod\limits_{j=1}^{k-1}\bigl(1-\|\Phi(z_j)(p)\|_{X^\ast}\bigr)\left\|\Phi(z_k)(p)\right\|_{X^\ast}\\
&=\sum_{k=1}^{\infty}3\varepsilon_k+\sum_{k=1}^{\infty}\Biggl(\prod\limits_{j=1}^{k-1}\bigl(1-\|\Phi(z_j)(p)\|_{X^\ast}\bigr)-\prod\limits_{j=1}^{k}\bigl(1-\|\Phi(z_j)(p)\|_{X^\ast}\bigr)\Biggr)\\
&<\frac{\xi}{48C}+2.
\end{align*}
Thus, for all $p,q\in M$ we have
\begin{align*}
\sum_{k=1}^{\infty}\frac{\left|\varphi_kz_k(p)-\varphi_kz_k(q)\right|}{\n{p-q}}&\leq\sum_{k=1}^\infty\lint{0}{1}{\Bigl|\Big\langle\frac{p-q}{\n{p-q}},\Phi(\varphi_kz_k)(q+t(p-q))\Big\rangle\Bigr|}{t}\\
&\leq\sum_{k=1}^{\infty}\lint{0}{1}{\bigl\|\Phi(\varphi_kz_k)(q+t(p-q))\bigr\|_{X^\ast}}{t}\\
&\leq\lint{0}{1}{\sum_{k=1}^{\infty}\bigl\|\Phi(\varphi_kz_k)(q+t(p-q))\bigr\|_{X^\ast}}{t}\\
&<\frac{\xi}{48C}+2.
\end{align*}
Now, fix any $\mu\in\free{M}$ and pick sequences $(a_i)_{i=1}^\infty\in\ell_1$ and $(p_i,q_i)_{i=1}^\infty\subset\widetilde M$ so that 
$$\mu=\sum_{i=1}^{\infty}a_i\frac{\delta(p_i)-\delta(q_i)}{\n{p_i-q_i}}.$$ We have
\begin{align*}
\sum_{k=1}^{\infty}\bigl|\big\langle\mu,\varphi_kz_k\res{M}\!\big\rangle\bigr|&\leq\sum_{k=1}^{\infty}\sum_{i=1}^{\infty}|a_i|\frac{\left|\varphi_kz_k(p_i)-\varphi_kz_k(q_i)\right|}{\n{p_i-q_i}}\\
&=\sum_{i=1}^{\infty}|a_i|\sum_{k=1}^{\infty}\frac{\left|\varphi_k z_k(p_i)-\varphi_k z_k(q_i)\right|}{\n{p_i-q_i}}\\
&<\n{(a_i)_{i=1}^\infty}_{\ell_1}\left(\frac{\xi}{48C}+2\right)<\infty.
\end{align*}
By virtue of the~Banach--Steinhaus uniform boundedness principle and Goldstine's theorem, we conclude that the~series $\sum_{k=1}^\infty \varphi_kz_k\res{M}$ is WUC.

In order to complete the~proof, we will show that $\sup_{\mu\in\Gamma}\abs{\ip{\mu}{\varphi_kz_k\res{M}}}\geq\xi/8$ for each $k\in\N$. Recall that for each $k\in\N$ the~measure $\kappa_k$ lies in $\Omega$, so it has a~fixed representation \eqref{eq: representation of mu} satisfying \eqref{norm of extension}. For simplicity, we relabel the~corresponding parameters as $I_k$, $a_i^{k}$, $p_i^{k}$ and $q_i^{k}$ ($1\leq i\leq I_k$). Observe that from definition~(\ref{definition of nu}) and conditions \ref{norm of z_k} and \ref{complicated inequality} it follows that
\begin{align}
\label{application of Bart Simpson}
\nonumber
\big|&\big\langle(1-\Psi_k)\Phi(z_k),\lambda_k\big\rangle\big|\leq\\\nonumber
&\leq\sum_{i=1}^{I_{k}} \left|a^{k}_i\right|\left\|p^{k}_i-q^{k}_i\right\|\lint{0}{1}{\left \| \left(1-\Psi_k\left(q^{k}_i+t\left(p^{k}_i-q^{k}_i\right)\right)\right)\Phi(z_k)\left(q^{k}_i+t\left(p^{k}_i-q^{k}_i\right)\right)\right\|_{X^\ast}}{t}\\\nonumber
&\leq\sum_{i=1}^{I_{k}} \left|a^{k}_i\right|\left\|p^{k}_i-q^{k}_i\right\|\lint{0}{1}{\big(1-\Psi_k\left(q^{k}_i+t\left(p^{k}_i-q^{k}_i\right)\right)\big)}{t}\\\nonumber
&=\sum_{i=1}^{I_{k}} \left|a^{k}_i\right|\left\|p^{k}_i-q^{k}_i\right\|\lint{0}{1}{\Big(1-\prod\limits_{j=1}^{k-1}\bigl(1-\n{\Phi(z_j)(q^{k}_i+t(p^{k}_i-q^{k}_i))}_{X^\ast}\bigr)\Big)}{t}\\\nonumber
&\leq\sum_{i=1}^{I_{k}} \left|a^{k}_i\right|\left\|p^{k}_i-q^{k}_i\right\|\lint{0}{1}{\sum_{j=1}^{k-1}\left\|\Phi(z_j)\left(q^{k}_i+t\left(p^{k}_i-q^{k}_i\right)\right)\right\|_{X^\ast}}{t}\\\nonumber
&\leq\sum_{j=1}^{k-1}\sum_{i=1}^{I_{k}} \left|a^{k}_i\right|\left\|p^{k}_i-q^{k}_i\right\|\lint{0}{1}{\left\|\Phi(z_j)\left(q^{k}_i+t\left(p^{k}_i-q^{k}_i\right)\right)\right\|_{X^\ast}}{t}\\
&<\sqrt{2C}\sum_{j=1}^{k-1}\varepsilon_j.
\end{align}
Note that in the~fifth line we used the~elementary inequality $1-\prod_{j=1}^{k}(1-\alpha_j)\leq\sum_{j=1}^{k}\alpha_j$ for $(\alpha_j)_{j=1}^{k}\subset[0,1]$.
Next, by combining \ref{kappa_k at z_k}, (\ref{application of Bart Simpson}), (\ref{norm of nu depends on convM}), \ref{close element} and \eqref{eq:sum epsilon}, we infer that
\begin{align*}
\abs{\langle \kappa_k,\varphi_kz_k\rangle} &=\left|\ip{\Phi(\varphi_kz_k)}{\lambda_k}\right|\\[2mm]
&\geq\left|\ip{\kappa_k}{z_k}\right|-\left|\ip{\left(1-\Psi_k\right)\Phi(z_k)}{\lambda_k}\right|-\left|\ip{\Phi(\varphi_kz_k)-\Psi_k\Phi(z_k)}{\lambda_k}\right|\\
&\geq\frac {\xi}{3}-\sqrt{2C}\sum_{j=1}^{k-1}\varepsilon_j-2C3\varepsilon_k\geq \frac{\xi}{4}.
\end{align*}
Finally, in view of \ref{close element} and \ref{norm of z_k}, we obtain
\begin{align*}
\n{\varphi_kz_k\res{M}}_{\Lipp}&\leq \sup_{p\in\overline{\conv}M}\n{\Phi(\varphi_kz_k)(p)}_{X^\ast}\\
&\leq\sup_{p\in\overline{\conv}M}\n{\Phi(\varphi_kz_k)(p)-\Psi_k(p)\Phi(z_k)(p)}_{X^\ast}\\
&\quad +\sup_{p\in\overline{\conv}M}\n{\Psi_k(p)\Phi(z_k)(p)}_{X^\ast}\\
&\leq3\varepsilon_k+1<\frac{\xi}{48C}+1.
\end{align*}
Thus, condition (\ref{mu close to gamma}) yields that there exists $\eta_k\in \Gamma$ such that $\abs{\ip{\eta_k}{(\varphi_kz_k)\res{M}}}\geq\xi/8$. The~proof is complete by defining $\varrho_k=(\varphi_kz_k)\res{M}$ for $k\in\N$.
\end{proof}

\begin{proof}[Proof of Main Theorem]
Of course, we can assume that $M$ contains the~origin of $X$ and that it is the~distinguished point in $M$. Let $(\mu_n)_{n=1}^{\infty}\subset\free{M}$ be a~weakly Cauchy sequence which is not weakly convergent. Then, by Theorem \ref{thm: technical theorem}, there is a~WUC series $\sum_{k=1}^\infty \varrho_k$ in $\Lipz{M}$ such that 
$$\limsup_{k\to\infty}\,\sup\bigl\{\abs{\ip{\mu_n}{\varrho_k}}\colon n\in \N\bigr\}>0.$$ Therefore, we can define a~bounded linear operator $\map{T}{\free{M}}{\ell_1}$ by 
$$T(\mu)=(\ip{\mu}{\varrho_k})_{k=1}^{\infty}\quad (\mu\in\free{M}),$$
so that $T(\{\mu_n\colon n\in\N\})$ is not relatively norm-compact in $\ell_1$. Hence, there exists a~subsequence of $(\mu_n)_{n=1}^{\infty}$ equivalent to the~unit vector basis of $\ell_1$ (see {\it e.g.} \cite[Thm.~III.C.9]{W_book}), which is a~contradiction with $(\mu_n)_{n=1}^\infty$ being weakly Cauchy.
\end{proof}

\section{Examples}\noindent
Below, we provide several examples of metric spaces to which our Main Theorem applies, and which were not covered by previously known results.

\vspace*{2mm}\noindent
{\bf 1. }For $p\in (1,\infty)$ let $\mathcal{Q}_p=\prod_{n=1}^\infty [0,\frac{1}{n}]$ be the Hilbert cube equipped with the $\ell_p$-metric, that is,
$$
\rho(\boldsymbol{\mathrm{x}},\boldsymbol{\mathrm{y}})=\Bigl(\sum_{n=1}^\infty\abs{x_n-y_n}^p\Bigr)^{1/p}\quad\mbox{for }\boldsymbol{\mathrm{x}}=(x_n)_{n=1}^\infty,\, \boldsymbol{\mathrm{y}}=(y_n)_{n=1}^\infty\in \mathcal{Q}_p.
$$
Plainly, $\mathcal{Q}_p$ is a compact subset of $\ell_p$ and hence the Main Theorem implies that for each $p\in(1,\infty)$ the Lipschitz-free space $\free{\mathcal{Q}_p}$ is weakly sequentially complete.

It is worth noticing that in this way we obtain a~collection of metric spaces which are mutually nonbilipschitz homeomorphic. To see this, we shall recall the~notion of metric type introduced by Enflo (\cite{enflo_type1}, \cite{enflo_type2}) and developed later in various forms (see {\it e.g.} \cite{BMW}). A~metric space $(M,\rho)$ has {\it Enflo type} $p$ if there exists a~constant $T>0$ such that for every $n\in\N$ and every map $f\colon \{-1,1\}^n\to M$ we have
\begin{equation*}
\begin{split}
\mathbb{E}_\e\, \bigl[\rho(f(\e),f(-\e))^p\bigr]\leq T^p\sum_{j=1}^n\mathbb{E}_\e\,\bigl[\rho\bigl( & f(\e_1,\ldots,\e_{j-1},\e_j,\e_{j+1},\ldots,\e_n),\\
& f(\e_1,\ldots,\e_{j-1},-\e_j,\e_{j+1},\ldots,\e_n)\bigr)^{\! p}\bigr],
\end{split}
\end{equation*}
where the expectation values are taken with respect to uniform choice of $\e\in\{-1,1\}^n$. Note that at the~left-hand side we have lengths of diagonals, whereas at the right-hand side we have lengths of edges of an~$n$-cube in $M$ determined by the~function $f$. As it was shown by Enflo \cite{enflo_type1}, $L_p([0,1])$ has Enflo type $p$ for every $p\in [1,2]$, and hence so does $\mathcal{Q}_p$.

Let $q\in (1,2]$, $n\in\N$ and consider a~map $f\colon\{-1,1\}^n\to\mathcal{Q}_q$ given by
$$
f(\e)=\Bigl(\frac{1+\e_1}{2n},\ldots,\frac{1+\e_n}{2n},0,0,\ldots\Bigr).
$$
Obviously, the length of each edge equals $n^{-1}$ and the lenght of each diagonal equals $n^{(1-q)/q}$. Therefore, if $\mathcal{Q}_q$ had Enflo type $p>q$, there would be a~constant  $C>0$ such that $n^{p/q}\leq Cn$ for every $n\in\N$, which is impossible. Since the Enflo type is a~bilipschitz invariant, we conclude that $\mathcal{Q}_q$ does not bilipschitz embed in $\mathcal{Q}_p$ for $1<q<p\leq 2$. In particular, the~metric spaces $\{\mathcal{Q}_p\colon 1<p\leq 2\}$ are mutually nonbilipschitz homeomorphic. The cases where $q>\max\{p,2\}$ or $2<q<p$ (in which it is known that $L_q$ does not bilipschitz embed in $L_p$) are more subtle, as seeking for metric invariants which would explain the corresponding nonembeddability results for $L_p$-spaces proved to be a~very difficult problem (see \cite{naor_schechtman} and the~references therein).


\vspace*{2mm}\noindent
{\bf 2. }Lafforgue and Naor \cite{LN} constructed, for each $p\in (2,\infty)$, a~doubling subset $\mathcal{M}_p$ of $L_p$ which does not admit a~bilipschitz embedding into $L_q$, for any $q\in (1,p)$. Recall that a~metric space $M$ is called {\it doubling} if for some $k\in\N$, every ball in $M$ can be covered by at most $k$ balls of half its radius, which obviously implies that every ball in $M$ is compact. Although the Lafforgue--Naor spaces $\mathcal{M}_p$'s are not compact, as being built with the~aid of a~`disjoint union argument' (see \cite[p.~388]{LN}), we can employ Kalton's theorem \cite[Prop.~4.3]{kalton} which gives
$$
\free{\mathcal{M}_p}\xhookrightarrow[\,1+\e\,]{}\Biggl(\bigoplus_{k=1}^\infty\free{\mathcal{M}_{p,k}}\Biggr)_{\!\!\ell_1},
$$
where $\mathcal{M}_{p,k}$ stands for the ball of radius $2^k$ centered at the~origin and the~arrow indicates a~$(1+\e)$-isometric linear embedding. Since weak sequential completeness is preserved by $\ell_1$-sums, we infer that for every $p\in (2,\infty)$ the~space $\free{\mathcal{M}_p}$ is weakly sequentially complete. In this way we have shown that the Main Theorem applies to a~class of noncompact metric spaces which are not bilipschitz embeddable into a~Hilbert space.

\vspace*{2mm}\noindent
{\bf 3. }Finally, let us mention that for a~certain class of metric spaces $M$ there are convenient conditions verifying whether $M$ bilipschitz embeds into an~$L_p$-space. Recall that if $K$ is any set, then a~map $f\colon K\times K\to\C$ is called {\it positive-definite} if 
$$
\sum_{1\leq i,j\leq n} f(t_i,t_j)\xi_i\overline{\xi_j}\geq 0
$$
for all $n\in\N$, $t_1,\ldots,t_n\in K$ and $\xi_1,\ldots,\xi_n\in\C$. By the classical Schoenberg's theorem \cite{schoenberg}, a~metric space $(M,\rho)$ isometrically embeds into a~Hilbert space if and only if the map $M\times M\ni (x,y)\mapsto\rho(x,y)^2$ is negative-definite on $M$; equivalently: $K_t(x,y)=\exp(-t\rho(x,y)^2)$ defines a~positive-definite map on $M$ for each $t>0$. Schoenberg also showed that for every $p\in [1,2]$ the map $\n{x-y}^p$ is negative-definite on $L_p$. Bretagnolle, Dacunha-Castelle and Krivine \cite{BD-CK} proved the converse, namely, if $(X,\n{\cdot})$ is a~normed space such that, for some $p\in [1,2]$, the map $\n{x-y}^p$ is negative-definite, then $X$ embeds linearly and isometrically into $L_p$. Consequently, if $M$ is a~compact subset of a~normed space $(X,\n{\cdot})$ with $\n{x-y}^p$ negative-definite for some $p\in (1,2]$, then $\free{M}$ is weakly sequentially complete.

\medskip\noindent {\bf Acknowledgement.}{ The authors would like to thank Gilles Godefroy for pointing out a connection between the obtained result and the study of property ($X$), as presented in Problem \ref{problem_X} and the note above.}

\end{document}